\documentclass[11pt,reqno]{amsart}

\usepackage{amsmath,amssymb,amsthm,mathtools}
\usepackage[margin=1in]{geometry}
\usepackage{microtype}
\usepackage{enumerate}
\usepackage[hidelinks]{hyperref}
\usepackage{float}
\usepackage{upgreek}
\usepackage{listings}
\usepackage{dsfont}

\newtheorem{innercustomgeneric}{\customgenericname}
\providecommand{\customgenericname}{}
\newcommand{\newcustomtheorem}[2]{%
  \newenvironment{#1}[1]
  {%
   \renewcommand\customgenericname{#2}%
   \renewcommand\theinnercustomgeneric{##1}%
   \innercustomgeneric
  }
  {\endinnercustomgeneric}
}

\newcustomtheorem{thm1}{Theorem}
\newcustomtheorem{lemma1}{Lemma}

\newtheorem{theorem}{Theorem}[section]
\newtheorem{thm}[theorem]{Theorem}
\newtheorem{lemma}[theorem]{Lemma}

\DeclareMathOperator{\stab}{Stab}
\DeclareMathOperator{\fl}{fl}
\DeclareMathOperator{\Sp}{Sp}

\DeclareMathOperator{\Aut}{Aut}
\DeclareMathOperator{\Gal}{Gal}
\DeclareMathOperator{\GF}{GF}
\DeclareMathOperator{\Irr}{Irr}

\title{Orbits of $p$-Nilpotent Group Actions for small primes}

\author{Chloe So}
\address{BASIS Independent Silicon Valley, 1290 Parkmoor Avenue, San Jose, CA 95126}
\email{sochloekaylee@gmail.com}

\author{Yong Yang}
\address{Department of Mathematics, Texas State University, 601 University Drive, San Marcos, TX 78666, USA}
\email{yang@txstate.edu}

\date{\today}

\begin{document}

\begin{abstract}
Let $G$ be a finite solvable group and $V$ be a faithful completely reducible $G$-module. We improve previous bounds for the sizes of defect groups of blocks of solvable groups through an improved bound on vector stabilizers of $p$-nilpotent groups. Specifically, if $G$ is 2-nilpotent and $2 \nmid |V|$, then there exists $v \in V$ such that $\mathbf{C}_G(v) \leq \mathbf{F}_4(G)$. Similarly, if $G$ is 3-nilpotent and $3 \nmid |V|$, then there exists $v \in V$ such that $\mathbf{C}_G(v) \leq \mathbf{F}_5(G)$. Using these results, we conclude that $d(B) \leq \lfloor \frac{9n}{10} \rfloor$ when $O_p(G) = 1$ for $p = 2$ and $d(B) \leq \lfloor \frac{19n}{21} \rfloor$ when $O_p(G) = 1$ for $p = 3$.
\end{abstract}

\maketitle

\section{Introduction}
Large orbit theorems are an important tool to study the structure of finite solvable groups through linear group actions. Consider a faithful completely reducible $G$-module $V$ for finite solvable group $G$. In general, elements in $V$ with large orbits have small stabilizers. For example, when $V$ has a regular orbit, there exists $v \in V$ such that $\mathbf{C}_G(v)=1$ by the Orbit-Stabilizer Theorem.

We can bound these stabilizers with the Fitting series of $G$. In \cite{Moreto}, Moreto and Wolf show that there exists $v \in V$ such that $\mathbf{C}_G(v) \leq \mathbf{F}_9(G)$. Yang further improves the bound to $\mathbf{F}_7(G)$ in \cite{Yang}. This result has important applications regarding character degrees and conjugacy classes.

Specifically, we focus on the defect groups of $p$-blocks in finite solvable groups. In \cite{Addendum}, Xiao et al. improved the bound on $d(B)$ when $O_p(G) = 1$ for $p = 2, 3$ by showing that if a $p$-group $A$ acts faithfully and coprimely on a solvable group $G$, then $|A : A'|$ is bounded by $b^8$ where $b$ is the largest orbit length of $A$ on $\Irr(G)$. This argument uses bounds for stabilizers of solvable group actions.

Motivated by the 2-nilpotent or 3-nilpotent condition of $A$ for $p = 2, 3$, respectively, it is natural to consider the stabilizer bound of solvable group actions for $p$-nilpotent groups. In this paper, we find that if $V$ is a finite completely reducible faithful $G$-module where $2 \nmid |V|$ for a finite solvable 2-nilpotent group $G$, then there exists $v \in V$ such that $\mathbf{C}_G(v) \leq \mathbf{F}_4(G)$. Similarly, if $V$ is a finite completely reducible faithful $G$-module where $3 \nmid |V|$ for a finite solvable 3-nilpotent group $G$, then there exists $v \in V$ such that $\mathbf{C}_G(v) \leq \mathbf{F}_5(G)$. Finally, these results are used to improve the bound of $d(B)$ from $\lfloor \frac{15n}{16} \rfloor$ to $\lfloor \frac{9n}{10} \rfloor$ when $p=2$ and from $\lfloor \frac{13n}{14} \rfloor$ to $\lfloor \frac{19n}{21} \rfloor$ when $p=3$.
\section{Preliminaries}

We begin by fixing our notation.

\begin{itemize}
    \item $G$ is $p$-nilpotent if there exists normal subgroup $K \triangleleft G$ such that $p \nmid |K|$ and $|G/K|=p^n$ for some non-negative integer $n$.
    \item Let $\mathbf{F}(G)$ be the Fitting subgroup of $G$.
    \item Let $\mathbf{F}(G) = F_1(G) \leq F_2(G) \leq \cdots \leq F_n(G) = G$ be the ascending Fitting series, where
    \[
    F_{i+1}(G)/F_i(G)=\mathbf{F}(G/F_i(G)).
    \]
    The Fitting length of $G$ is $\fl(G)=n$. We call $F_1(G)$ the bottom term of the Fitting series and $G/F_{n-1}(G)$ the top Fitting factor of $G$.
    \item If $G$ is a linear group on $V$, let $r(G)$ be the number of orbits of $G$ on $V$.
    \item Let $V$ be a Galois field $\GF(q^n)$ for a prime power $q$. Define a semilinear group to be
    \begin{equation}
        \Gamma(q^n) = \{ x \mapsto ax^{\sigma} \mid a \in \GF(q^n)^\#, \sigma \in \Gal(\GF(q^n)/\GF(q)) \}.
    \end{equation}
    \item If $H$ is a linear group and $S$ is a permutation group, $H \wr S$ denotes the wreath product of $H$ with $S$.
\end{itemize}

We now introduce a few lemmas.

\begin{lemma}
    If $S$ is a solvable primitive group on $\Omega$, $S$ has a unique minimal normal subgroup $M$ with $|M| = |\Omega| = q^n$ for prime $q$, $S = MS_\alpha$ where $\alpha \in \Omega$, and $M$ is a faithful irreducible $S_\alpha$-module of order $q^n$. In particular, $\fl(S) = 1 + \fl(S_\alpha)$.
\end{lemma}

\begin{proof}
    Since $S$ is solvable and primitive, every nontrivial normal subgroup of $S$ is transitive. A minimal normal subgroup $M$ is therefore elementary abelian and regular, so $|M|=|\Omega|=q^n$ for some prime $q$. Writing $S=MS_\alpha$, the conjugation action of $S_\alpha$ on $M$ is faithful and irreducible. Since $M=\mathbf F(S)$, the equality $\fl(S)=1+\fl(S_\alpha)$ follows.
\end{proof}

\begin{lemma}
    If $H$ is a finite solvable group and $S$ is a finite solvable permutation group, then
    \[
        \fl(H \wr S) \leq \fl(H)+\fl(S).
    \]
\end{lemma}

\begin{proof}
    Let $B$ be the base group of $H\wr S$. Then $B$ is a direct product of copies of $H$, so $\fl(B)=\fl(H)$, while $(H\wr S)/B\cong S$. The standard extension inequality for Fitting length gives
    \[
        \fl(H\wr S)\leq \fl(B)+\fl(S)=\fl(H)+\fl(S).
    \]
\end{proof}

\begin{lemma}
    Suppose $V$ is a faithful quasi-primitive $G$-module where $|V|=q^n$ for a solvable group $G$. Note that $q$ is the order of the base field and $n=\dim(V)$. Then, $G$ has normal subgroups $U \leq T \leq F \leq A \leq G$ such that

    \begin{enumerate}[ {(}1{)} ]
        \item $F = \mathbf{F}(G)$ is a central product $F=ET$ with $E, T \triangleleft G$. $E$ is a direct product of extraspecial groups with $|E/Z| = e^2$ where $Z = E \cap T = \mathbf{Z}(E) = Soc(\mathbf{Z}(F))$.
        \item $F/T$ is a direct sum of irreducible $G/F$-modules $W_i$.
        \item $|W_i| = e_i^2$ for a prime power $e_i$. Also, $e=e_1 \cdots e_s$ divides $n$ and $\gcd(q,e) = 1$.
        \item $A = \mathbf{C}_G(Z)$ and $A/F$ acts faithfully on $F/T$.
        \item $A/\mathbf{C}_A(E_i/Z) \leq \Sp(2n_i, p_i)$ where $e_i = p_i^{n_i}$.
        \item $G/A$ is abelian since $G/A \leq \Aut(Z)$ and even $G=A$ when $e=n$.
        \item $|T:U| \leq 2$, where $U$ is cyclic and acts fixed point freely on $W$ where $W$ is an irreducible submodule of $V_U$.
        \item If $e=1$, then $G \leq \Gamma(q^n)$ and $\fl(G) \leq 2$.
    \end{enumerate}
\end{lemma}

\begin{proof}
    This follows from Corollaries 1.10 and 2.6 of \cite{Manz/Wolf}.
\end{proof}

\begin{lemma}
    If $N$ is normal in $G$, then $\mathbf{F}_i(N) \leq \mathbf{F}_i(G)$ for all $i$.
\end{lemma}

\begin{proof}
    The assertion for $i=1$ follows from $\mathbf F(N)\leq \mathbf F(G)$ when $N\triangleleft G$. Assume $\mathbf F_i(N)\leq\mathbf F_i(G)$. The image of $\mathbf F_{i+1}(N)$ in $N\mathbf F_i(G)/\mathbf F_i(G)$ is contained in the Fitting subgroup of that quotient, which in turn is contained in $\mathbf F(G/\mathbf F_i(G))$. Hence $\mathbf F_{i+1}(N)\leq\mathbf F_{i+1}(G)$, and induction completes the proof.
\end{proof}

\section{2-Nilpotent Group Actions}

\begin{lemma}
    Suppose that $G$ is a solvable 2-nilpotent irreducible subgroup of GL$(n,q)$ for a prime power $q$, where $2 \nmid q$. If $n=2$ or $(n,q)=(4,3)$, then $\fl(G) \leq 2$.
\end{lemma}

\begin{proof}
    Let $V$ an irreducible $G$-module with $|V|=q^n$. We use a pair $(n,q)$ to represent the different cases. 
    
    If $n=2$, by Theorem 2.11 of \cite{Manz/Wolf}, one of the following cases must occur:
    \begin{enumerate}[ {(}a{)} ]
        \item $G \leq Z_{q-1} \wr Z_2$;
        \item $G \leq \Gamma(q^2)$; or
        \item $\mathbf{F}(G) = QT$ where $Q_8 \cong Q \trianglelefteq G, T = \mathbf{Z}(\mathbf{F}(G)) = \mathbf{Z}(G)$ is cyclic, $T \leq \mathbf{Z}(GL(2,q)), T \cap Q = \mathbf{Z}(Q)$ and $G/\mathbf{F}(G) \cong Z_3$ or $S_3$. Also $q \neq 2$.
    \end{enumerate}

    We now analyze each case. In (a), $\fl(G) \leq 1+1 = 2$. In (b), $\fl(G) \leq 2$. In (c), $G$ cannot be 2-nilpotent. Therefore, $\fl(G) \leq 2$.

    If $(n,q)=(4,3)$, we can use GAP to manually check the Fitting length of every 2-nilpotent subgroup (see Appendix). Through this, we find $\fl(G) \leq 2$.
\end{proof}

\begin{lemma}
    Suppose that $G$ is a 2-nilpotent completely reducible solvable subgroup of GL$(V)$ for a vector space $V \ne \{0\}$ where $2 \nmid |V|$. Let $r$ be the number of orbits of $G$ on $V$, so that $r>1$. If $r \leq 3$ then $\fl(G) \leq r$. Additionally, if $\fl(G) = 3$, $G \leq H \wr Z_2$, where $H \leq \Gamma(q^n)$.
\end{lemma}

\begin{proof}
    If $G$ were not irreducible, $r \geq 2^2 = 4$. Therefore, let $G$ is a 2-nilpotent irreducible subgroup of GL$(n,q)$ for a prime power $q$ where $|V| = q^n$. Since $V \ne \{0\}$, $r>1$.
    
    If $G$ is a semilinear group, then $\fl(G) \leq 2$. Assume $G$ is not a subgroup of a semilinear.
    
    When $r=2$, $G$ acts transitively on the non-zero vectors of $V$. By Theorem 6.8 of \cite{Manz/Wolf}, $G$ must be semilinear or a few exceptional cases, all of which are not 2-nilpotent. This is a contradiction.
    
    Consider $r=3$. If $V$ is imprimitive, then $G$ is isomorphic to a subgroup of $H \wr S$. If we let $s=r(H)$ and $S$ be a permutation group on $m$ letters, then $r(G) \geq {m+s-1\choose m}$ by Lemma 2.6 of \cite{Foulser}. If $s \geq 3$ or if $m \geq 3$, then $G$ has at least 4 orbits on $V$. We thus assume $s=2$ and $m=2$. Since $s=2$, then $\fl(H) \leq 2$ by the previous paragraph. We then have $\fl(G) \leq \fl(H) + \fl(Z_2) \leq 3$.
    
    If $V$ is primitive, by Theorem 1.1 of \cite{Foulser}, $G$ must satisfy one of 4 cases:
    \begin{enumerate}[ {(}a{)} ]
        \item $(n,q)=(3,4)$; because $2 \nmid q$, we disregard this case.
        \item $(n,q)=(2,p^{2k})$; by Lemma 3.1, $\fl(G) \leq 2$.
        \item $(n,q)=(4,3)$; $\fl(G) \leq 2$ by our GAP code (see Appendix).
        \item $(n,q)=(4,7)$ with $G_0$ containing orbits of length 480 and 1980; $\fl(G) \leq 2$ by our GAP code (see Appendix).
    \end{enumerate}
    Therefore, $\fl(G) \leq 2$.

    % Now, suppose $G$ is not irreducible. We induct on $\dim(V)$. Write $V = V_1 \oplus V_2$ for completely reducible modules $V_1$ and $V_2$. Denote $C_i = \mathbf{C}_G(V_i)$. Then, $G$ is isomorphic to a subgroup of $G/C_1 \times G/C_2$. By the inductive hypothesis, $\fl(G/C_1) \leq 3$ and $\fl(G/C_2) \leq 3$, so $\fl(G) \leq 3$.
\end{proof}

\begin{thm}
    Let $V$ be a faithful irreducible quasi-primitive $G$-module where $2 \nmid |V|$ for a finite solvable 2-nilpotent group $G$. Then, $G$ has at least 1 regular orbit on $V$ or $\fl(G) \leq 4$.
\end{thm}

\begin{proof}
    First, by \cite{HoltYang}, if $e = 5,6,7,$ or $e \geq 10$ and $e \ne 16$, $G$ has at least one regular orbit $v$ on $V$. If $e=1$, $\fl(G) \leq 2$, so we only need to consider $e=2,3,4,8,9,16$.

    If $e=2^n$, then each $E_i$ is a 2-group. Because $G$ is 2-nilpotent, there exists a normal Hall $2'$-subgroup $N$ of $G$. Then $[E_i,N]\leq E_i\cap N=1$, so the image of $G$ on $W_i$ is a 2-group. A finite 2-group acting on a nonzero vector space over $\mathds F_2$ has a nonzero fixed vector; since $W_i$ is irreducible, the action is trivial. Thus $G/\mathbf C_G(W_i)=1$. Otherwise $e=3$ or $9$, and Lemma 3.1 gives
    \[
        \fl\bigl(G/\mathbf C_G(W_i)\bigr)\leq 2.
    \]
    Since $F/T=\bigoplus_i W_i$ and $\mathbf C_G(F/T)=\bigcap_i\mathbf C_G(W_i)$, the diagonal map gives
    \[
        G/\mathbf C_G(F/T)\hookrightarrow \prod_i G/\mathbf C_G(W_i).
    \]
    As $A/F$ acts faithfully on $F/T$, it follows that
    \[
        \fl(A/F)\leq \fl\bigl(G/\mathbf C_G(F/T)\bigr)
        \leq \max_i \fl\bigl(G/\mathbf C_G(W_i)\bigr)\leq 2.
    \]
    Hence $\fl(G)\leq \fl(F)+\fl(A/F)+\fl(G/A)\leq 1+2+1=4$.
\end{proof}

\begin{thm}
    Suppose that $H$ is a linear group on $V_1$ and $S$ is a solvable primitive permutation group on $\Omega$, with $|\Omega|=m$. Then $V = V_1 + V_2 + \cdot \cdot \cdot + V_m$ is an $H \wr S$-module where $V_i$ are copies of $V_1$. Let $G$ be a 2-nilpotent subgroup of $H \wr S$ whose projection onto $S$ is $S$.

    \begin{enumerate}
        \item Suppose $H$ has at least 2 orbits on $V_1$ such that $C_H(v_1) \leq L, L \triangleleft H$. If also $m > 9$ then $G$ has at least 4 orbits on $V$ such that $\mathbf{C}_G(v) \leq (L \times L \times \cdots \times L) \cap G$.
        \item Suppose $H$ has at least 3 orbits on $V_1$ such that $C_H(v_1) \leq L, L \triangleleft H$. If also $m > 4$ then $G$ has at least 4 orbits on $V$ such that $\mathbf{C}_G(v) \leq (L \times L \times \cdots \times L) \cap G$.
        \item Suppose $H$ has at least 4 orbits on $V_1$ such that $C_H(v_1) \leq L, L \triangleleft H$. Then $G$ has at least 4 orbits on $V$ such that $\mathbf{C}_G(v) \leq (L \times L \times \cdots \times L) \cap G$.
    \end{enumerate}
\end{thm}

\begin{proof}
    $G$ is a subgroup of $H \wr S$, its projection onto $S$ is $S$, and $S$ is a solvable primitive permutation group on $\Omega$. Let $H$ have at least $s$ orbits of elements in $V_1$ whose centralizer in $H$ is in $L$ where $L \triangleleft H.$
    
    We call a partition "trivially centralized" when $\Omega$ can be partitioned into a disjoint union of non-empty subsets $\Delta_1,...\Delta_j$ and $2 \leq j \leq s$ such that $\stab_S(\Delta_1) \cap \cdots \cap \stab_S(\Delta_j) = 1$. Note that when a partition is trivially centralized, $\mathbf{C}_G(v) \leq (L \times L \times \cdots \times L) \cap G$. We refer to the number of orbits of $G$ on $V$ where this occurs as "good orbits."

    Suppose $S$ has a regular orbit on $\mathfrak{B}(S)$, the power set of $\Omega$. Then, there exists some $\Delta_1 \subset \Omega$ such that $\stab_S(\Delta_1) = 1$. Let $\Delta_2 = \Omega \setminus \Delta_1$. This partition is trivially centralized. If $m$ is odd, we have $s$ choices for the orbit on $\Delta_1$ and $s-1$ choices for the orbit on $\Delta_2$, so $G$ has at least $s(s-1)$ good orbits. If $m$ is even, $\Delta_1$ and $\Delta_2$ may have the same number of elements; then, $G$ has at least $\frac{s(s-1)}{2}$ good orbits. Therefore, we have:

    \begin{enumerate}
        \item If $s \geq 4$, $G$ has at least $\frac{4 \cdot 3}{2} > 4$ good orbits.
        \item If $s=3$ and $m > 4$, $G$ has at least $4$ good orbits by \cite{Yang}. Additionally, if $m=3$, $G$ has at least $3 \cdot 2 > 4$ good orbits. Therefore, if $s=3$ and $m\ne 2,4$, $G$ has at least 4 good orbits.
        \item If $s=2$ and $m>9$, $G$ has at least $4$ good orbits by \cite{Yang}.
    \end{enumerate}

    Next, assume that $S$ does not have a regular orbit on $\mathfrak{B}(S)$. By Theorem 5.6 of \cite{Manz/Wolf}, we only have a few additional cases. Because $S$ must be 2-nilpotent, $S$ must be isomorphic to $S_3, F_{10}, F_{20}, F_{42}, $ or the semidirect product of $Z_3 \times Z_3$ with $D_8$ or $SD_{16}$.
    
    To calculate the number of orbits of $H \wr S$, we define the following terms. Denote $s=r(H)$. Let $\Pi$ be the set of all partitions of $\Omega$ and $\pi$ be an element of $\Pi$. Let $B(\pi) = n_1 + \cdots + n_j$, where $n_i$ is the number of blocks in $\pi$ with the $i$th largest size, and $F(\pi) = n_1! \cdots n_j!$. Define $N(\pi)$ as the number of orbits of $S$ on $\pi$. By \cite{Permutation}, the number of orbits can be calculated as follows:

    \begin{equation}
        r(H \wr S) = \sum_{\pi \in \Pi} N(\pi) \cdot \frac{P(s, B(\pi))}{F(\pi)}.
    \end{equation}

    To count the number of good orbits, we restrict $N(\pi)$ to only count the orbits of trivially centralized partitions. Then, using the computer program GAP (see Appendix), we find the following:
    
    \begin{table}[H]
        \centering
        \resizebox{5cm}{!}{
        \begin{tabular}{cccc}
            \hline
            $B(\pi)$ & Type & $N(\pi)$ & $F(\pi)$ \\
            \hline
            3 & $(1,1,1)$ & 1 & 6 \\
            \hline
        \end{tabular}
        }
        \caption{Calculations for $S_3$}
        \label{tab:table1}
    \end{table}

    \begin{table}[H]
        \centering
        \resizebox{5cm}{!}{
        \begin{tabular}{cccc}
            \hline
            $B(\pi)$ & Type & $N(\pi)$ & $F(\pi)$ \\
            \hline
            5 & $(1,1,1,1,1)$ & 12 & 120 \\
            4 & $(1,1,1,2)$ & 6 & 6 \\
            3 & $(1,1,3)$ & 2 & 2 \\
            3 & $(1,2,2)$ & 2 & 2 \\
            \hline
        \end{tabular}
        }
        \caption{Calculations for $F_{10}$}
        \label{tab:table2}
    \end{table}

    \begin{table}[H]
        \centering
        \resizebox{5cm}{!}{
        \begin{tabular}{cccc}
            \hline
            $B(\pi)$ & Type & $N(\pi)$ & $F(\pi)$ \\
            \hline
            5 & $(1,1,1,1,1)$ & 6 & 120 \\
            4 & $(1,1,1,2)$ & 3 & 6 \\
            3 & $(1,1,3)$ & 1 & 2 \\
            3 & $(1,2,2)$ & 1 & 2 \\
            \hline
        \end{tabular}
        }
        \caption{Calculations for $F_{20}$}
        \label{tab:table3}
    \end{table}

    \begin{table}[H]
        \centering
        \resizebox{5cm}{!}{
        \begin{tabular}{cccc}
            \hline
            $B(\pi)$ & Type & $N(\pi)$ & $F(\pi)$ \\
            \hline
            7 & $(1,1,1,1,1,1,1)$ & 120 & 5040 \\
            6 & $(1,1,1,1,1,2)$ & 60 & 120 \\
            5 & $(1,1,1,1,3)$ & 20 & 24 \\
            5 & $(1,1,1,2,2)$ & 30 & 12 \\
            4 & $(1,1,1,4)$ & 5 & 6 \\
            4 & $(1,1,2,3)$ & 10 & 2 \\
            4 & $(1,2,2,2)$ & 14 & 6 \\
            3 & $(1,1,5)$ & 1 & 2 \\
            3 & $(1,2,4)$ & 2 & 1 \\
            3 & $(1,3,3)$ & 3 & 2 \\
            3 & $(2,2,3)$ & 4 & 2 \\
            \hline
        \end{tabular}
        }
        \caption{Calculations for $F_{42}$}
        \label{tab:table4}
    \end{table}

    \begin{table}[H]
        \centering
        \resizebox{6cm}{!}{
        \begin{tabular}{cccc}
            \hline
            $B(\pi)$ & Type & $N(\pi)$ & $F(\pi)$ \\
            \hline
            9 & $(1,1,1,1,1,1,1,1,1)$ & 5040 & 362880 \\
            8 & $(1,1,1,1,1,1,1,2)$ & 2520 & 5040 \\
            7 & $(1,1,1,1,1,2,2)$ & 1260 & 240 \\
            6 & $(1,1,1,2,2,2)$ & 624 & 36 \\
            5 & $(1,2,2,2,2)$ & 300 & 24 \\
            7 & $(1,1,1,1,1,1,3)$ & 840 & 720 \\
            6 & $(1,1,1,1,2,3)$ & 420 & 24 \\
            5 & $(1,1,2,2,3)$ & 204 & 4 \\
            4 & $(2,2,2,3)$ & 92 & 6 \\
            5 & $(1,1,1,3,3)$ & 140 & 12 \\
            4 & $(1,2,3,3)$ & 64 & 2 \\
            3 & $(3,3,3)$ & 18 & 6 \\
            6 & $(1,1,1,1,1,4)$ & 210 & 120 \\
            5 & $(1,1,1,2,4)$ & 102 & 6 \\
            4 & $(1,2,2,4)$ & 46 & 2 \\
            4 & $(1,1,3,4)$ & 32 & 2 \\
            3 & $(2,3,4)$ & 12 & 1 \\
            3 & $(1,4,4)$ & 6 & 2 \\
            5 & $(1,1,1,1,5)$ & 42 & 24 \\
            4 & $(1,1,2,5)$ & 18 & 2 \\
            3 & $(2,2,5)$ & 6 & 2 \\
            3 & $(1,3,5)$ & 4 & 1 \\
            4 & $(1,1,1,6)$ & 6 & 6 \\
            3 & $(1,2,6)$ & 2 & 1 \\
            \hline
        \end{tabular}
        }
        \caption{Calculations for $(Z_3 \times Z_3) \rtimes D_8$}
        \label{tab:table5}
    \end{table}

    \begin{table}[H]
        \centering
        \resizebox{6cm}{!}{
        \begin{tabular}{cccc}
            \hline
            $B(\pi)$ & Type & $N(\pi)$ & $F(\pi)$ \\
            \hline
            9 & $(1,1,1,1,1,1,1,1,1)$ & 2520 & 362880 \\
            8 & $(1,1,1,1,1,1,1,2)$ & 1260 & 5040 \\
            7 & $(1,1,1,1,1,2,2)$ & 630 & 240 \\
            6 & $(1,1,1,2,2,2)$ & 312 & 36 \\
            5 & $(1,2,2,2,2)$ & 150 & 24 \\
            7 & $(1,1,1,1,1,1,3)$ & 420 & 720 \\
            6 & $(1,1,1,1,2,3)$ & 210 & 24 \\
            5 & $(1,1,2,2,3)$ & 102 & 4 \\
            4 & $(2,2,2,3)$ & 46 & 6 \\
            5 & $(1,1,1,3,3)$ & 70 & 12 \\
            4 & $(1,2,3,3)$ & 32 & 2 \\
            3 & $(3,3,3)$ & 9 & 6 \\
            6 & $(1,1,1,1,1,4)$ & 105 & 120 \\
            5 & $(1,1,1,2,4)$ & 51 & 6 \\
            4 & $(1,2,2,4)$ & 23 & 2 \\
            4 & $(1,1,3,4)$ & 16 & 2 \\
            3 & $(2,3,4)$ & 6 & 1 \\
            3 & $(1,4,4)$ & 3 & 2 \\
            5 & $(1,1,1,1,5)$ & 21 & 24 \\
            4 & $(1,1,2,5)$ & 9 & 2 \\
            3 & $(2,2,5)$ & 3 & 2 \\
            3 & $(1,3,5)$ & 2 & 1 \\
            4 & $(1,1,1,6)$ & 3 & 6 \\
            3 & $(1,2,6)$ & 1 & 1 \\
            \hline
        \end{tabular}
        }
        \caption{Calculations for $(Z_3 \times Z_3) \rtimes SD_{16}$}
        \label{tab:table6}
    \end{table}

    Applying the formula, we find $r(H \wr S) \geq 4$ in all exceptional cases except for $S_3$, which only gives at least 4 good orbits for $s \geq 4$.
\end{proof}

\begin{thm}
    Let $V$ be a faithful irreducible $G$-module where $2 \nmid |V|$ for a finite solvable 2-nilpotent group $G$. There exists $K \triangleleft G$ such that $\fl(K) \leq 4$ and $G$ has at least 1 nonzero orbit of elements $v \in V$ with $\mathbf{C}_G(x) \leq K$. Furthermore, if $G$ has less than 4 orbits of elements $v \in V$ with $\mathbf{C}_G(x) \leq K$, then $G$ has less than 4 orbits on $V$.
\end{thm}

\begin{proof}
    Suppose $V$ is a faithful quasi-primitive $G$-module. If $G$ has at least 1 regular orbit on $V$, then we choose $K=1$. Otherwise, let $K=G$. By Theorem 3.3, we have $\fl(K) = \fl(G) \leq 4$.

    Suppose $V$ is not quasi-primitive; then, there exists $N \triangleleft G$ such that $V_N = V_1 \oplus \cdots \oplus V_m$ for $m > 1$ homogeneous components $V_i$ of $V_N$. Let $M$ be the kernel of the permutation action; then $S=G/M$ is the primitive permutation group induced by $G$ on $\{V_1,\ldots,V_m\}$. Consider $V_1$ and let $H=\mathbf N_G(V_1)/\mathbf C_G(V_1)$. Then $H$ acts faithfully and irreducibly on $V_1$, and the standard imprimitive embedding identifies $G$ with a subgroup of $H\wr S$ whose projection onto $S$ is $S$. Choose two distinct $H$-orbits $C_1,D_1$ on $V_1$. For each $i$, choose $g_i\in G$ with $V_1^{g_i}=V_i$ and put $C_i=C_1^{g_i}$ and $D_i=D_1^{g_i}$; these sets are independent of the choice of $g_i$. Choose $x_i\in C_i$ and $y_i\in D_i$. If
    \[
        v=x_1+\cdots+x_j+y_{j+1}+\cdots+y_m,
    \]
    then every element of $\mathbf C_G(v)$ preserves the two orbit types and hence its image in $S$ stabilizes both $\{V_1,\ldots,V_j\}$ and $\{V_{j+1},\ldots,V_m\}$.

    Now, $S$ must be a solvable primitive permutation group on $\Omega = \{V_1, \dots, V_m\}$. By induction on $\dim(V)$, either there exists $L \triangleleft H$ with $\fl(L) \leq 4$ such that $H$ has at least 4 orbits of elements in $V_1$ with centralizers in $L$, or $H$ has less than 4 orbits on $V_1$. If $H$ has at least 4 orbits of elements in $V_1$ with centralizers in $L$, let $K = (L \times \cdots \times L) \cap G$. By Lemma 3.4(3), $G$ has at least 4 orbits on $V$ such that $\mathbf{C}_G(x) \leq K$.

    If $H$ has less than 4 orbits on $V_1$. By Lemma 3.4, $G$ has at least 4 orbits on $V$ such that $\mathbf{C}_G(x) \leq (H \times \cdots \times H) \cap G$ when $r(H)=2$ and $m>9$, or when $r(H)=3$ and $m>4$. Then, let $K = (H \times \cdots \times H) \cap G$. This leaves us with a few more cases:
    
    If $r(H) = 2$, $m \leq 9$. Then, $\fl(H) \leq 2$ by Lemma 3.2. Because $S$ must be 2-nilpotent, $\fl(S) \leq 2$, so $\fl(G) \leq \fl (H) + \fl(S) = 4$ by Lemma 2.2. Then, let $K=G$.

    If $r(H) = 3$, $m \leq 4$. Then, $\fl(H) \leq 3$ by Lemma 3.2. If $m=2$, $\fl(S) = 1$, so $\fl(G) \leq \fl (H) + \fl(S) = 4$ by Lemma 2.2.
    
    Assume $m=3$. If $\fl(H)=2$, $\fl(G) \leq \fl (H) + \fl(S) = 4$. Else, $H \leq L \wr Z_2$, where $L \leq \Gamma(q^n)$. We have $G \leq H \wr S_3 \leq (L \wr Z_2)\wr (A_3 \rtimes Z_2)$. Let $J \leq G$. Because $J$ is 2-nilpotent, there exists $N \triangleleft J$ such that $2 \nmid |N|$ and $J/N$ is a 2-group. 
    
    Then, because $N$ has odd order, the image of $N$ onto the $Z_2$ in $S_3$ is trivial, and the image of $N$ onto the $Z_2^3$ in $(L \wr Z_2)^3$ is trivial. Therefore, $N \leq L^3 \rtimes A_3 = L \wr A_3$. We conclude that $\fl(J) \leq \fl(N) + \fl(J/N) \leq \fl(L) + \fl(A_3) + \fl(J/N) \leq 2+1+1=4$.
    
    Notice that there exists no primitive 2-nilpotent permutation groups of degree 4, so we disregard the case $m=4$. Thus, for $r(H) = 3$, let $K=G$.
\end{proof}

\begin{thm}
    If $V$ is a finite completely reducible faithful $G$-module where $2 \nmid |V|$, possibly of mixed characteristic, for a solvable 2-nilpotent group $G$, then there exists $v \in V$ and $K \triangleleft G$ such that $\mathbf{C}_G(v) \leq K$ and $\fl(K) \leq 4$.
\end{thm}

\begin{proof}
    We induct on $|V|$. If $V$ is irreducible, the statement is true by Theorem 3.5. Then, write $V = V_1 \oplus V_2$ for completely reducible modules $V_1$ and $V_2$. Denote $C_i = \mathbf{C}_G(V_i)$. Then, $G$ is isomorphic to a subgroup of $G/C_1 \times G/C_2$. By the inductive hypothesis, there exists $v_i$ and $K_i \triangleleft (G/C_i)$ such that $\mathbf{C}_{G/C_i}(v_i) \leq K_i$ and $\fl(K_i) \leq 4$. Let $\pi:G\longrightarrow G/C_1\times G/C_2$ be the diagonal map and set
    \[
        K=\pi^{-1}\bigl((K_1\times K_2)\cap\pi(G)\bigr).
    \]
    Then $K\triangleleft G$. For $v=v_1+v_2$, we have $\mathbf C_G(v)\leq K$. Since $V$ is faithful, $C_1\cap C_2=1$, so $\pi$ is injective and $K$ embeds in $K_1\times K_2$. Therefore
    \[
        \fl(K)\leq\max\{\fl(K_1),\fl(K_2)\}\leq4.
    \]

\end{proof}

\begin{thm}
    If $V$ is a finite completely reducible faithful $G$-module where $2 \nmid |V|$, possibly of mixed characteristic, for a finite solvable 2-nilpotent group $G$, then there exists $v \in V$ such that $\mathbf{C}_G(v) \leq \mathbf{F}_4(G)$.
\end{thm}

\begin{proof}
    By Theorem 3.6, there exists $v \in V$ and $K \triangleleft G$ such that $\mathbf{C}_G(v) \leq K$ and $\fl(K) \leq 4$. By Lemma 2.4, because $K \triangleleft G$ and $\fl(K) \leq 4$, $\mathbf{C}_G(v) \leq K \leq \mathbf{F}_4(G)$.
\end{proof}

\section{3-Nilpotent Group Actions}

\begin{lemma}
    Suppose that $G$ is a solvable 3-nilpotent irreducible subgroup of GL$(n,q)$ for a prime power $q$ where $3 \nmid q$.

    \begin{enumerate}[ {(}1{)} ]
        \item If $(n,q) = (2,q)$, then $\fl(G) \leq 2$.
        \item If $(n,q) = (3,4)$, then $\fl(G) \leq 3$.
        \item If $(n,q) = (4,2)$, then $\fl(G) \leq 2$.
        \item If $(n,q) = (6,2)$, then $\fl(G) \leq 3$.
        \item If $(n,q) = (8,2)$, then $\fl(G) \leq 3$.
    \end{enumerate}
\end{lemma}

\begin{proof}
    Let $V$ an irreducible $G$-module with $|V|=q^n$. We use a pair $(n,q)$ to represent the different cases. 
    
    Consider when $(n,q) = (2,q)$. By Theorem 2.11 of \cite{Manz/Wolf}, one of the following cases must occur:
    \begin{enumerate}[ {(}a{)} ]
        \item $G \leq Z_{q-1} \wr Z_2$;
        \item $G \leq \Gamma(q^2)$; or
        \item $\mathbf{F}(G) = QT$ where $Q_8 \cong Q \trianglelefteq G, T = \mathbf{Z}(\mathbf{F}(G)) = \mathbf{Z}(G)$ is cyclic, $T \leq \mathbf{Z}(GL(2,q)), T \cap Q = \mathbf{Z}(Q)$ and $G/\mathbf{F}(G) \cong Z_3$ or $S_3$. Also $q \neq 2$.
    \end{enumerate}
    
    We now analyze each case. In (a), $\fl(G) \leq 1+1 = 2$. In (b), $\fl(G) \leq 2$. In (c), $G/\mathbf{F}(G)$ cannot be $S_3$, as $S_3$ is not 3-nilpotent. Therefore, $\fl(G) \leq \fl(\mathbf{F}(G)) + \fl(Z_3) = 1+1 = 2$.

    Consider when $(n,q) = (3,4)$. By Theorem 2.12 of \cite{Manz/Wolf}, one of the following cases must occur:
    \begin{enumerate}[ {(}a{)} ]
        \item $G \leq Z_{3} \wr S$ where $S \leq S_3$;
        \item $G \leq \Gamma(4^3)$; or
        \item $G/\mathbf{F}(G) \leq SL(2,3)$.
    \end{enumerate}
    
    We now analyze each case. In (a), $\fl(G) \leq 1+2 = 3$. In (b), $\fl(G) \leq 2$. In (c), $\fl(G) \leq \fl(\mathbf{F}(G)) + \fl(G/\mathbf{F}(G)) = 1+2 = 3$.

    Consider when $(n,q) = (4,2)$. By Theorem 2.15 of \cite{Manz/Wolf}, one of the following cases must occur:
    \begin{enumerate}[ {(}a{)} ]
        \item $G \leq S_3 \wr Z_2$; or
        \item $G \leq \Gamma(2^4)$;
    \end{enumerate}
    
    We now analyze each case. In (a), let $B=S_3\times S_3$ be the base group. Then $G\cap B\triangleleft G$ is 3-nilpotent, and each coordinate projection of $G\cap B$ is a 3-nilpotent subgroup of $S_3$, hence nilpotent. Thus $G\cap B$ is nilpotent. Since $G/(G\cap B)$ is a subgroup of $Z_2$, we obtain $\fl(G)\leq2$. In (b), $\fl(G) \leq 2$.

    Consider when $(n,q) = (6,2)$. By Theorem 2.15 of \cite{Manz/Wolf}, one of the following cases must occur:
    \begin{enumerate}[ {(}a{)} ]
        \item $G \leq \Gamma(2^3) \wr Z_2$;
        \item $G \leq S_3 \wr Z_3 \cdot Z_2$;
        \item $G \leq \Gamma(2^6)$; or
        \item $G/\mathbf{F}(G) \leq$ GL$(2,3)$.
    \end{enumerate}
    
    We now analyze each case. In (a), $\fl(G) \leq \fl(\Gamma(2^3))+\fl(Z_2)=3$. In (b), let $B$ be the base group. The coordinate projections of the 3-nilpotent group $G\cap B$ are 3-nilpotent subgroups of $S_3$, hence nilpotent; therefore $G\cap B$ is nilpotent. The quotient $G/(G\cap B)$ has Fitting length at most 2, so $\fl(G)\leq3$. In (c), $\fl(G) \leq 2$. In (d), $\fl(G) \leq 2 + 1 = 3$.

    Consider when $(n,q) = (8,2)$. 

    If $V$ is a faithful primitive $G$-module for the solvable group $G$, then $V$ is quasi-primitive. Applying Lemma 2.3, we have that $e \mid 8$ and $(e,q)=1$, so $e=1$. Then, $\fl(G) \leq 2$.

    Suppose $V$ is not primitive. Then, $G$ is isomorphic to a subgroup of $H \wr S$, where $S$ is a primitive permutation group of degree $m$ for an integer $m > 1$ with $mt = n$ and $H$ is a solvable irreducible subgroup of GL$(t,q)$. By Lemma 2.1, $m$ is a prime power. Then, $G$ is isomorphic to a subgroup of one of the following:

    \begin{enumerate}[ {(}1{)} ]
        \item $H \wr S_2$, where $H$ is a solvable irreducible subgroup of GL$(4,2)$. By the case $(n,q)=(4,2)$ above, $\fl(H) \leq 2$, so $\fl(G) \leq \fl (H) + \fl(S_2) = 3$.
        \item $H \wr S$, where $H$ is a solvable irreducible subgroup of GL$(2,2)$ and $S$ is a primitive permutation group of degree 4. Notice that $H$ must be $Z_2$ or $Z_3$. The induced permutation group $S$ is a 3-nilpotent subgroup of $S_4$, so $\fl(S)\leq2$. Hence $\fl(G)\leq\fl(H)+\fl(S)\leq3$.
        \item $H \wr S$, where $H$ is a solvable irreducible subgroup of GL$(1,2)$ and $S$ is a primitive permutation group of degree 8. If $\fl(S) \leq 2$, $\fl(G) \leq 3$. Notice that $\fl(S)>2$ only when $S=A\Gamma(2^3)$, where $\fl(S)=3$. Since $\operatorname{GL}(1,2)=1$, we have $H=1$, and therefore $\fl(G)\leq\fl(S)=3$.
    \end{enumerate}
\end{proof}

\begin{lemma}
    Suppose that $G$ is a 3-nilpotent completely reducible solvable subgroup of GL$(V)$ for a vector space $V \ne \{0\}$ where $3 \nmid q$. Let $r$ be the number of orbits of $G$ on $V$, so that $r>1$. If $r \leq 3$ then $\fl(G) \leq r$.
\end{lemma}

\begin{proof}
     If $G$ were not irreducible, $r \geq 2^2 = 4$. Therefore, let $G$ is a 3-nilpotent irreducible subgroup of GL$(n,q)$ for a prime power $q$ where $|V| = q^n$ and $3 \nmid q$. Since $V \ne \{0\}$, $r>1$.
    
    If $G$ is a semilinear group, then $\fl(G) \leq 2$.
    
    When $r=2$, $G$ acts transitively on the non-zero vectors of $V$. By Theorem 6.8 of \cite{Manz/Wolf}, one of few exceptional cases occur. However, the only 3-nilpotent cases where $q \ne 3$ have $G/\mathbf{F}(G)=Z_3$, so $\fl(G) = 2$.
    
    Now we assume $r=3$.
    
    If $V$ is imprimitive, then $G$ is isomorphic to a subgroup of $H \wr S$. If we let $s=r(H)$ and $S$ be a permutation group on $m$ letters, then $r(G) \geq {m+s-1\choose m}$ by Lemma 2.6 of \cite{Foulser}. If $s \geq 3$ or if $m \geq 3$, then $G$ has at least 4 orbits on $V$. We thus assume $s=2$ and $m=2$. Since $s=2$, then $\fl(H) \leq 2$ by the previous paragraph. We then have $\fl(G) \leq \fl(H) + \fl(Z_2) \leq 3$.
    
    If $V$ is primitive, by Theorem 1.1 of \cite{Foulser}, $G$ must satisfy one of 4 cases:
    \begin{enumerate}[ {(}a{)} ]
        \item $(n,q)=(3,4)$; by Lemma 4.1, $\fl(G) \leq 3$.
        \item $(n,q)=(2,p^{2k})$; by Lemma 4.1, $\fl(G) \leq 2$.
        \item $(n,q)=(4,3)$; because $3 \nmid q$, we disregard this case.
        \item $(n,q)=(4,7)$ with $G_0$ containing orbits of length 480 and 1980; $\fl(G) \leq 3$ by our GAP code (see Appendix).
    \end{enumerate}
    Therefore, $\fl(G) \leq 3$.

    % Now, suppose $G$ is not irreducible. We induct on $\dim(V)$. Write $V = V_1 \oplus V_2$ for completely reducible modules $V_1$ and $V_2$. Denote $C_i = \mathbf{C}_G(V_i)$. Then, $G$ is isomorphic to a subgroup of $G/C_1 \times G/C_2$. By the inductive hypothesis, $\fl(G/C_1) \leq 3$ and $\fl(G/C_2) \leq 3$, so $\fl(G) \leq 3$.
\end{proof}

\begin{thm}
    Let $V$ be a faithful irreducible quasi-primitive $G$-module with $3 \nmid |V|$ for a finite solvable 3-nilpotent group $G$, then $G$ has at least 1 regular orbit on $V$ or $\fl(G) \leq 5$.
\end{thm}

\begin{proof}
    First, by \cite{HoltYang}, if $e = 5,6,7,$ or $e \geq 10$ and $e \ne 16$, $G$ has at least one regular orbit $v$ on $V$. If $e=1$, $\fl(G) \leq 2$, so we only need to consider $e=2,3,4,8,9,16$.

    If $e=3^n$, then each $E_i$ is a 3-group. Because $G$ is 3-nilpotent, there exists a normal Hall $3'$-subgroup $N$ of $G$. Then $[E_i,N]\leq E_i\cap N=1$, so the image of $G$ on $W_i$ is a 3-group. A finite 3-group acting on a nonzero vector space over $\mathds F_3$ has a nonzero fixed vector; since $W_i$ is irreducible, the action is trivial. Thus $G/\mathbf C_G(W_i)=1$. Otherwise $e=2,4,8,$ or $16$, and Lemma 4.1 gives
    \[
        \fl\bigl(G/\mathbf C_G(W_i)\bigr)\leq 3.
    \]
    Since $F/T=\bigoplus_i W_i$ and $\mathbf C_G(F/T)=\bigcap_i\mathbf C_G(W_i)$, the diagonal map gives
    \[
        G/\mathbf C_G(F/T)\hookrightarrow \prod_i G/\mathbf C_G(W_i).
    \]
    As $A/F$ acts faithfully on $F/T$, it follows that
    \[
        \fl(A/F)\leq \fl\bigl(G/\mathbf C_G(F/T)\bigr)
        \leq \max_i \fl\bigl(G/\mathbf C_G(W_i)\bigr)\leq 3.
    \]
    Hence $\fl(G)\leq \fl(F)+\fl(A/F)+\fl(G/A)\leq 1+3+1=5$.
\end{proof}

\begin{thm}
    Suppose that $H$ is a linear group on $V_1$ and $S$ is a solvable primitive permutation group on $\Omega$, with $|\Omega|=m$. Then $V = V_1 + V_2 + \cdot \cdot \cdot + V_m$ is an $H \wr S$-module where $V_i$ are copies of $V_1$. Let $G$ be a 3-nilpotent subgroup of $H \wr S$ whose projection onto $S$ is $S$.

    \begin{enumerate}
        \item Suppose $H$ has at least 2 orbits on $V_1$ such that $C_H(v_1) \leq L, L \triangleleft H$. If also $m > 9$ then $G$ has at least 4 orbits on $V$ such that $\mathbf{C}_G(v) \leq (L \times L \times \cdots \times L) \cap G$.
        \item Suppose $H$ has at least 3 orbits on $V_1$ such that $C_H(v_1) \leq L, L \triangleleft H$. If also $m \neq 2,4$ then $G$ has at least 4 orbits on $V$ such that $\mathbf{C}_G(v) \leq (L \times L \times \cdots \times L) \cap G$.
        \item Suppose $H$ has at least 4 orbits on $V_1$ such that $C_H(v_1) \leq L, L \triangleleft H$. Then $G$ has at least 4 orbits on $V$ such that $\mathbf{C}_G(v) \leq (L \times L \times \cdots \times L) \cap G$.
    \end{enumerate}
\end{thm}

\begin{proof}
    $G$ is a subgroup of $H \wr S$, its projection onto $S$ is $S$, and $S$ is a solvable primitive permutation group on $\Omega$. Let $H$ have at least $s$ orbits of elements in $V_1$ whose centralizer in $H$ is in $L$ where $L \triangleleft H.$
    
    We call a partition "trivially centralized" when $\Omega$ can be partitioned into a disjoint union of non-empty subsets $\Delta_1,...\Delta_j$ and $2 \leq j \leq s$ such that $\stab_S(\Delta_1) \cap \cdots \cap \stab_S(\Delta_j) = 1$. Note that when a partition is trivially centralized, $\mathbf{C}_G(v) \leq (L \times L \times \cdots \times L) \cap G$. We refer to the number of orbits of $G$ on $V$ where this occurs as "good orbits."

    Suppose $S$ has a regular orbit on $\mathfrak{B}(S)$, the power set of $\Omega$. Then, there exists some $\Delta_1 \subset \Omega$ such that $\stab_S(\Delta_1) = 1$. Let $\Delta_2 = \Omega \setminus \Delta_1$. This partition is trivially centralized. If $m$ is odd, we have $s$ choices for the orbit on $\Delta_1$ and $s-1$ choices for the orbit on $\Delta_2$, so $G$ has at least $s(s-1)$ good orbits. If $m$ is even, $\Delta_1$ and $\Delta_2$ may have the same number of elements; then, $G$ has at least $\frac{s(s-1)}{2}$ good orbits. Therefore, we have:

    \begin{enumerate}
        \item If $s \geq 4$, $G$ has at least $\frac{4 \cdot 3}{2} > 4$ good orbits.
        \item If $s=3$ and $m > 4$, $G$ has at least $4$ good orbits by \cite{Yang}. Additionally, if $m=3$, $G$ has at least $3 \cdot 2 > 4$ good orbits. Therefore, if $s=3$ and $m\ne 2,4$, $G$ has at least 4 good orbits.
        \item If $s=2$ and $m>9$, $G$ has at least $4$ good orbits by \cite{Yang}.
    \end{enumerate}

    Next, assume that $S$ does not have a regular orbit on $\mathfrak{B}(S)$. By Theorem 5.6 of \cite{Manz/Wolf}, we only have a few additional cases. Because $S$ must be 3-nilpotent, $S$ must be isomorphic to $A_4, F_{10}, F_{20}, F_{42}, $ or $A\Gamma(2^3)$. Notice that we have already considered $F_{10}, F_{20},$ and $F_{42}$ during the 2-nilpotent case.
    
    To calculate the number of orbits of $H \wr S$, we define the following terms. Denote $s=r(H)$. Let $\Pi$ be the set of all partitions of $\Omega$ and $\pi$ be an element of $\Pi$. Let $B(\pi) = n_1 + \cdots + n_j$, where $n_i$ is the number of blocks in $\pi$ with the $i$th largest size, and $F(\pi) = n_1! \cdots n_j!$. Define $N(\pi)$ as the number of orbits of $S$ on $\pi$. By \cite{Permutation}, the number of orbits can be calculated as follows:

    \begin{equation}
        r(H \wr S) = \sum_{\pi \in \Pi} N(\pi) \cdot \frac{P(s, B(\pi))}{F(\pi)}.
    \end{equation}

    To count the number of good orbits, we restrict $N(\pi)$ to only count the orbits of trivially centralized partitions. Then, using the computer program GAP (see Appendix), we find the following:
    
    \begin{table}[H]
        \centering
        \resizebox{5cm}{!}{
        \begin{tabular}{cccc}
            \hline
            $B(\pi)$ & Type & $N(\pi)$ & $F(\pi)$ \\
            \hline
            4 & $(1,1,1,1)$ & 2 & 24 \\
            3 & $(1,1,2)$ & 1 & 2 \\
            \hline
        \end{tabular}
        }
        \caption{Calculations for $A_4$}
        \label{tab:table7}
    \end{table}

    \begin{table}[H]
        \centering
        \resizebox{5cm}{!}{
        \begin{tabular}{cccc}
            \hline
            $B(\pi)$ & Type & $N(\pi)$ & $F(\pi)$ \\
            \hline
            8 & $(1,1,1,1,1,1,1,1)$ & 240 & 40320 \\
            7 & $(1,1,1,1,1,1,2)$ & 120 & 720 \\
            6 & $(1,1,1,1,2,2)$ & 60 & 48 \\
            5 & $(1,1,2,2,2)$ & 30 & 12 \\
            4 & $(2,2,2,2)$ & 14 & 24 \\
            6 & $(1,1,1,1,1,3)$ & 40 & 120 \\
            5 & $(1,1,1,2,3)$ & 20 & 6 \\
            4 & $(1,2,2,3)$ & 10 & 2 \\
            4 & $(1,1,3,3)$ & 6 & 4 \\
            3 & $(2,3,3)$ & 3 & 2 \\
            5 & $(1,1,1,1,4)$ & 10 & 24 \\
            4 & $(1,1,2,4)$ & 5 & 2 \\
            3 & $(2,2,4)$ & 2 & 2 \\
            3 & $(1,3,4)$ & 1 & 1 \\
            4 & $(1,1,1,5)$ & 2 & 6 \\
            3 & $(1,2,5)$ & 1 & 1 \\
            \hline
        \end{tabular}
        }
        \caption{Calculations for $A\Gamma(2^3)$}
        \label{tab:table8}
    \end{table}

    Applying the formula, we find $r(H \wr S) \geq 4$ in all exceptional cases except for $A_4$, which only gives at least 4 good orbits for $s \geq 4$.
\end{proof}

\begin{thm}
    Let $V$ be a faithful irreducible $G$-module where $3 \nmid |V|$ for a finite solvable 3-nilpotent group $G$. There exists $K \triangleleft G$ such that $\fl(K) \leq 5$ and $G$ has at least 1 nonzero orbit of elements $v \in V$ with $\mathbf{C}_G(v) \leq K$. Furthermore, if $G$ has less than 4 orbits of elements $v \in V$ with $\mathbf{C}_G(v) \leq K$, then $G$ has less than 4 orbits on $V$.
\end{thm}

\begin{proof}
    Suppose $V$ is a faithful quasi-primitive $G$-module. If $G$ has at least 1 regular orbit on $V$, then we choose $K=1$. Otherwise, let $K=G$. By Theorem 4.3, we have $\fl(K) = \fl(G) \leq 5$.

    Suppose $V$ is not quasi-primitive; then, there exists $N \triangleleft G$ such that $V_N = V_1 \oplus \cdots \oplus V_m$ for $m > 1$ homogeneous components $V_i$ of $V_N$. Let $M$ be the kernel of the permutation action; then $S=G/M$ is the primitive permutation group induced by $G$ on $\{V_1,\ldots,V_m\}$. Consider $V_1$ and let $H=\mathbf N_G(V_1)/\mathbf C_G(V_1)$. Then $H$ acts faithfully and irreducibly on $V_1$, and the standard imprimitive embedding identifies $G$ with a subgroup of $H\wr S$ whose projection onto $S$ is $S$. Choose two distinct $H$-orbits $C_1,D_1$ on $V_1$. For each $i$, choose $g_i\in G$ with $V_1^{g_i}=V_i$ and put $C_i=C_1^{g_i}$ and $D_i=D_1^{g_i}$; these sets are independent of the choice of $g_i$. Choose $x_i\in C_i$ and $y_i\in D_i$. If
    \[
        v=x_1+\cdots+x_j+y_{j+1}+\cdots+y_m,
    \]
    then every element of $\mathbf C_G(v)$ preserves the two orbit types and hence its image in $S$ stabilizes both $\{V_1,\ldots,V_j\}$ and $\{V_{j+1},\ldots,V_m\}$.

    Now, $S$ must be a solvable primitive permutation group on $\Omega = \{V_1, \dots, V_m\}$. By induction on $\dim(V)$, either there exists $L \triangleleft H$ with $\fl(L) \leq 5$ such that $H$ has at least 4 orbits of elements in $V_1$ with centralizers in $L$, or $H$ has less than 4 orbits on $V_1$. If $H$ has at least 4 orbits of elements in $V_1$ with centralizers in $L$, let $K = (L \times \cdots \times L) \cap G$. By Lemma 4.4(3), $G$ has at least 4 orbits on $V$ such that $\mathbf{C}_G(x) \leq K$.

    If $H$ has less than 4 orbits on $V_1$. By Lemma 4.4, $G$ has at least 4 orbits on $V$ such that $\mathbf{C}_G(x) \leq (H \times \cdots \times H) \cap G$ when $r(H)=2$ and $m>9$, or when $r(H)=3$ and $m \ne 2,4$. Then, let $K = (H \times \cdots \times H) \cap G$. This leaves us with a few more cases:
    
    If $r(H) = 2$, $m \leq 9$. Then, $\fl(H) \leq 2$ by Lemma 4.2. Because $S$ must be 3-nilpotent, $\fl(S) \leq 3$, so $\fl(G) \leq \fl (H) + \fl(S) = 5$ by Lemma 2.2. Then, let $K=G$.

    If $r(H) = 3$, $m = 2, 4$. Then, $\fl(H) \leq 3$ by Lemma 4.2. Because $S$ must be 3-nilpotent, $\fl(S) \leq 2$, so $\fl(G) \leq \fl (H) + \fl(S) = 5$ by Lemma 2.2. Then, let $K=G$.
\end{proof}

\begin{thm}
    If $V$ is a finite completely reducible faithful $G$-module, possibly of mixed characteristic where $3 \nmid |V|$, for a solvable 3-nilpotent group $G$, then there exists $v \in V$ and $K \triangleleft G$ such that $\mathbf{C}_G(v) \leq K$ and $\fl(K) \leq 5$.
\end{thm}

\begin{proof}
    We induct on $|V|$. If $V$ is irreducible, the statement is true by Theorem 4.5. Then, write $V = V_1 \oplus V_2$ for completely reducible modules $V_1$ and $V_2$. Denote $C_i = \mathbf{C}_G(V_i)$. Then, $G$ is isomorphic to a subgroup of $G/C_1 \times G/C_2$. By the inductive hypothesis, there exists $v_i$ and $K_i \triangleleft (G/C_i)$ such that $\mathbf{C}_{G/C_i}(v_i) \leq K_i$ and $\fl(K_i) \leq 5$. Let $\pi:G\longrightarrow G/C_1\times G/C_2$ be the diagonal map and set
    \[
        K=\pi^{-1}\bigl((K_1\times K_2)\cap\pi(G)\bigr).
    \]
    Then $K\triangleleft G$. For $v=v_1+v_2$, we have $\mathbf C_G(v)\leq K$. Since $V$ is faithful, $C_1\cap C_2=1$, so $\pi$ is injective and $K$ embeds in $K_1\times K_2$. Therefore
    \[
        \fl(K)\leq\max\{\fl(K_1),\fl(K_2)\}\leq5.
    \]

\end{proof}

\begin{thm}
    If $V$ is a finite completely reducible faithful $G$-module where $3 \nmid |V|$, possibly of mixed characteristic, for a finite solvable 3-nilpotent group $G$, then there exists $v \in V$ such that $\mathbf{C}_G(v) \leq \mathbf{F}_5(G)$.
\end{thm}

\begin{proof}
    By Theorem 4.6, there exists $v \in V$ and $K \triangleleft G$ such that $\mathbf{C}_G(v) \leq K$ and $\fl(K) \leq 5$. By Lemma 2.4, because $K \triangleleft G$ and $\fl(K) \leq 5$, $\mathbf{C}_G(v) \leq K \leq \mathbf{F}_5(G)$.
\end{proof}

\section{Sharpness of the Results}

Here, we show that the bounds in Theorems 3.7 and 4.7 are best possible.

First consider the case $p=2$. Let
\[
H=\Gamma L(1,3^5)=C_{242}\rtimes C_5
\]
act on the natural module $W=\GF(3^5)$, viewed as a $5$-dimensional
vector space over $\GF(3)$. Since $242=2\cdot 11^2$ and the Frobenius
automorphism $x\mapsto x^3$ fixes $-1$, we have
\[
H\cong C_2\times (C_{121}\rtimes C_5).
\]
In particular, $H$ has a normal Hall $2'$-subgroup, and it is
transitive on $W^\#$. Let
\[
G=H\wr S_3=H^3\rtimes S_3
\]
act on $V=W^3$ in the natural imprimitive action. Thus $|V|=3^{15}$,
and $V$ is a faithful irreducible $G$-module: its three coordinate
summands are pairwise nonisomorphic irreducible modules for $H^3$,
and $S_3$ permutes them transitively. The product of the normal Hall
$2'$-subgroup of $H^3$ with $A_3$ is a normal Hall $2'$-subgroup
of $G$. Hence $G$ is $2$-nilpotent. Furthermore,
\[
\mathbf F_1(G)=C_{242}^3,\qquad
\mathbf F_2(G)=H^3,\qquad
\mathbf F_3(G)=H^3\rtimes A_3,
\]
and $\mathbf F_4(G)=G$. For $v=(v_1,v_2,v_3)\in V$, put
\[
\Delta(v)=\{i\mid v_i\ne0\}\subseteq\{1,2,3\}.
\]
Since $H$ is transitive on $W^\#$, every element of the setwise
stabilizer of $\Delta(v)$ in $S_3$ lifts to an element of
$\mathbf C_G(v)$. Every subset of $\{1,2,3\}$ is stabilized by a
transposition. Hence $\mathbf C_G(v)$ contains an element whose
image in $S_3$ lies outside $A_3$, and therefore
\[
\mathbf C_G(v)\nleq \mathbf F_3(G)
\]
for every $v\in V$. Thus the bound $\mathbf F_4(G)$ in
Theorem 3.7 cannot be replaced by $\mathbf F_3(G)$.

Now consider the case $p=3$. Keep
\[
H=\Gamma L(1,2^5)=C_{31}\rtimes C_5
\]
on $W=\GF(2^5)$, and let
\[
S=A\Gamma L(1,2^3)=2^3\rtimes(C_7\rtimes C_3)
\]
act in its natural degree-$8$ action on $\GF(2^3)$. Set
\[
G=H\wr S=H^8\rtimes S
\]
and let $G$ act on $V=W^8$. Since $H$ is a $3'$-group and
\[
AGL(1,2^3)=2^3\rtimes C_7\triangleleft S
\]
is a normal Hall $3'$-subgroup of $S$, the subgroup
\[
H^8\rtimes AGL(1,2^3)
\]
is a normal Hall $3'$-subgroup of $G$. Thus $G$ is 3-nilpotent. Its Fitting series satisfies
\[
\mathbf F_1(G)=C_{31}^8,\qquad
\mathbf F_2(G)=H^8,
\]
\[
\mathbf F_3(G)/H^8\cong 2^3,\qquad
\mathbf F_4(G)/H^8\cong AGL(1,2^3),
\]
and $\mathbf F_5(G)=G$.

For $v=(v_1,\ldots,v_8)\in V$, identify the eight coordinates with $\GF(2^3)$ and put
\[
\Delta(v)=\{x\in \GF(2^3)\mid v_x\ne0\}.
\]
Again, every element of the setwise stabilizer of $\Delta(v)$ in $S$ lifts to an element of $\mathbf C_G(v)$. Let $\sigma:x\mapsto x^2$ be the Frobenius automorphism of $\GF(2^3)$. The group $AGL(1,2^3)$ has one orbit on the $k$-subsets of $\GF(2^3)$ for $k=0,1,2,3$, and two orbits on the $4$-subsets. If $\alpha$ is a generator of $\GF(2^3)^\#$, representatives for these orbits may be chosen among
\[
\varnothing,\quad \{0\},\quad \{0,1\},\quad
\{\alpha,\alpha^2,\alpha^4\},\quad
\{0,\alpha,\alpha^2,\alpha^4\},\quad
\{1,\alpha,\alpha^2,\alpha^4\}.
\]
Each of these subsets is fixed by $\sigma$, and complements give the corresponding statement for subsets of size greater than $4$. Hence every subset of $\GF(2^3)$ is stabilized by an element of
\[
A\Gamma L(1,2^3)\setminus AGL(1,2^3).
\]
It follows that, for every $v\in V$, the group $\mathbf C_G(v)$ contains an element whose image in $S$ does not lie in $AGL(1,2^3)$. Consequently,
\[
\mathbf C_G(v)\nleq \mathbf F_4(G)
\]
for every $v\in V$. Thus the bound $\mathbf F_5(G)$ in Theorem 4.7 cannot be replaced by $\mathbf F_4(G)$.

\section{Applications}
We use these new bounds on group actions to strengthen the theorems in \cite{Addendum}.

\begin{thm}
    Let $A$ be a 2-group that acts faithfully on solvable 2'-group $G$. Let $b$ be an integer such that $|A:C_A(\chi)| \leq b$ for all $\chi \in \Irr(G)$, then $|A:A'| \leq b^5$.
\end{thm}

\begin{proof}
    We follow the proof in Theorem 2.2 of \cite{Addendum}, which proves that when $A$ is a $p$-group that acts faithfully on solvable $p'$-group $G$, $|A:A'| \leq b^8$. Because $p=2$, $\Gamma = AG$ is 2-nilpotent. Thus, we can replace $\mathbf{F}_7(\Gamma/\mathbf{F}(\Gamma))=\mathbf{F}_8(\Gamma)/\mathbf{F}(\Gamma)$ with $\mathbf{F}_4(\Gamma/\mathbf{F}(\Gamma)) = \mathbf{F}_5(\Gamma)/\mathbf{F}(\Gamma)$ by our Theorem 3.7. It follows that $\mathbf{C}_\Gamma(\lambda) < \mathbf{F}_5(\Gamma)$ instead of $\mathbf{F}_8(\Gamma)$, which reduces the bound to $|A:A'| \leq b^5$.
\end{proof}

\begin{thm}
    Let $G$ be a finite solvable group such that $O_p(G) = 1$ for $p=2$ and let $|G|_p = p^n$. Then $G$ contains a $p$-block $B$ such that $d(B) \leq \lfloor \frac{9n}{10} \rfloor$.
\end{thm}

\begin{proof}
    We follow the proof in Theorem 2.3 of \cite{Addendum}, which proves that there exists $\mu \in \Irr(N_0)$ such that $|P/P' : C_{P/P'}(\mu)| \geq p^{n_1/8}$. Because $p=2$, we can replace $p^{n_1/8}$ with $p^{n_1/5}$ by our Theorem 5.1. Substituting $n_2 = n-n_1$ into $n_2 < n_1$, we find $n_1 > \frac{n}{2}$. Therefore, $\varphi(1)_p \geq p^{n_1/5} > p^{n/10}$, so $d(B) \leq \lfloor \frac{9n}{10} \rfloor$.
\end{proof}

\begin{thm}
    Let $A$ be a 3-group that acts faithfully on solvable 3'-group $G$. Let $b$ be an integer such that $|A:C_A(\chi)| \leq b$ for all $\chi \in \Irr(G)$, then $|A:A'| \leq b^6$.
\end{thm}

\begin{proof}
    We follow the proof in Theorem 2.2 of \cite{Addendum}, which proves that when $A$ is a $p$-group that acts faithfully on solvable $p'$-group $G$, $|A:A'| \leq b^8$. Because $p=3$, $\Gamma = AG$ is 3-nilpotent. Thus, we can replace $\mathbf{F}_7(\Gamma/\mathbf{F}(\Gamma))=\mathbf{F}_8(\Gamma)/\mathbf{F}(\Gamma)$ with $\mathbf{F}_5(\Gamma/\mathbf{F}(\Gamma)) = \mathbf{F}_6(\Gamma)/\mathbf{F}(\Gamma)$ by our Theorem 4.7. It follows that $\mathbf{C}_\Gamma(\lambda) < \mathbf{F}_6(\Gamma)$ instead of $\mathbf{F}_8(\Gamma)$, which reduces the bound to $|A:A'| \leq b^6$.
\end{proof}

\begin{thm}
    Let $G$ be a finite solvable group such that $O_p(G) = 1$ for $p=3$ and let $|G|_p = p^n$. Then $G$ contains a $p$-block $B$ such that $d(B) \leq \lfloor \frac{19n}{21} \rfloor$.
\end{thm}

\begin{proof}
    We follow the proof in Theorem 2.4 of \cite{Addendum}, which proves that there exists $\mu \in \Irr(N_0)$ such that $|P/P' : C_{P/P'}(\mu)| \geq p^{n_1/8}$. Because $p=3$, we can replace $p^{n_1/8}$ with $p^{n_1/6}$ by our Theorem 5.3. Substituting $n_2 = n-n_1$ into $n_2 < \frac{3n_1-2}{4}$, we find $n_1 > \frac{4n + 2}{7} > \frac{4n}{7}$. Therefore, $\varphi(1)_p \geq p^{n_1/6} > p^{2n/21}$, so $d(B) \leq \lfloor \frac{19n}{21} \rfloor$.
\end{proof}

\section*{Appendix}
All GAP code used in this paper is contained here: https://github.com/chloekso/math/tree/main/orbit

\section*{Acknowledgements}
This work was partially supported by a grant from the Simons Foundation (\#918096, to YY).

\section*{Disclosure Statement}
The authors declare that they have no competing interests and no conflicts of interest.

\section*{Data Availability Statement}
Data sharing is not applicable to this article, as no data sets were generated or analysed during the current study.

\end{document}